\documentclass[a4paper,oneside,11pt]{article} % Indique que le document est de type standard ("article"), au format A4 ("a4"), en recto simple ("oneside"), avec des fontes de taille moyenne ("11pt").

\usepackage[bottom]{footmisc}

\usepackage{a4wide}

\usepackage{amssymb}
\usepackage{tikz}
\usepackage{amsmath}

\usepackage{comment}

\usepackage{nameref}

\usepackage{bbm}

\usepackage{enumitem, hyperref}
\makeatletter
\def\namedlabel#1#2{\begingroup
    #2%
    \def\@currentlabel{#2}%
    \phantomsection\label{#1}\endgroup
}
\makeatother

\usepackage[all,cmtip]{xy}
\usepackage{tikz-cd}
\usepackage{graphicx}

\usepackage[all]{xy}

\usepackage{mathtools}
\DeclarePairedDelimiter\ev{\langle}{\rangle}

\usepackage[utf8]{inputenc} % Encodage des caractères du fichier-source.
\usepackage[T1]{fontenc} % Encodage des caractères en sortie.
\usepackage{textcomp} % Jeu de symboles complémentaires.
\usepackage[english]{babel}
\usepackage[autolanguage]{numprint} %% Formatage des nombres.
\usepackage{hyperref} % Génère des liens hypertexte dans le fichier pdf.
\usepackage{graphicx} % Permet l'insertion d'images.
\usepackage{verbatim} % Définit des environnements de texte préformaté et de commentaires.

\usepackage{amsmath,amssymb,amsthm} % Toutes ces extensions proviennent de la classe AMS-LaTeX.
\usepackage{comment}

\usepackage{fancyhdr} % Extension pour créer les en-têtes personnalisés.
\renewcommand\[{\begin{equation}}\renewcommand\]{\end{equation}} % Avec cette ligne, vos formules seront systématiquement numérotées.
\renewcommand\epsilon\varepsilon % Graphie du symbole "epsilon".
\renewcommand\phi\varphi % Graphie du symbole "phi".

\newtheorem{innercustomthm}{Theorem}
\newenvironment{customthm}[1]
  {\renewcommand\theinnercustomthm{#1}\innercustomthm}
  {\endinnercustomthm}
 
      \newcommand\NN{\mathcal{N}}
      \newcommand\TT{\mathcal{T}}
      \newcommand\ZZ{\mathbb{Z}} % Ensemble des entiers relatifs.
      \newcommand\QQ{\mathbb{Q}} % Ensemble des nombres rationnels.
      \newcommand\ab\allowbreak % Raccourci pour \allowbreak.
      
      \newcommand\GW{\operatorname{GW}}

      \newcommand\VV{\mathcal{V}}
      
      \newcommand\Spec{\operatorname{Spec}}

      \newcommand\M{M_{-1}}

      \newcommand\MQ{M_{-1,\QQ}}

      \newcommand\Tr{\operatorname{Tr}}
      \newcommand\tr{\operatorname{tr}}

      \newcommand\res{\operatorname{res}}

      \newcommand\Sm{\operatorname{\textbf{Sm}}_k}

      \newcommand\eeta{\boldsymbol{\eta}}
      \newcommand\Om{\operatorname{\Omega}}

      \newcommand\AAA{\mathbb{A}}

      \providecommand{\keywords}[1]
      {
        \small	
        \textbf{\textit{Keywords---}} #1
      }
      
      \providecommand{\Codes}[1]
      {
        \small	
        \textbf{\textit{MSC---}} #1
      }

      \newcommand\kMW{\mathbf{K}^{\text{MW}}}

      \newcommand\LL{\mathcal{L}}
      
      \newcommand\LLL{\omega}

      \newcommand\Gm{\mathbb{G}_m}

       \newcommand\PP{\mathbb{P}}

\theoremstyle{definition} % Pour tout ce qui ressemble à des définitions : définitions, notations...
\newtheorem{Def}{Definition}[subsection] % Définitions. Avec l'option "[section]", les définitions (et tout le reste) seront numérotées en fonction de la section courante.
\theoremstyle{plain} % Pour tous ce qui ressemble à des théorèmes.
\newtheorem{Pro}[Def]{Proposition} % Propositions. Avec l'option "[Def]", les propositions (et tout le reste) seront numérotées collectivement avec les définitions.
\newtheorem{Lem}[Def]{Lemma} % Lemmes.
\newtheorem{The}[Def]{Theorem} % Théorèmes.
\newtheorem{Cor}[Def]{Corollary} % Corollaires.
\theoremstyle{remark} % Pour tout ce qui ressemble à des remarques.
\newtheorem{Rem}[Def]{Remark} % Remarques.
\newtheorem{Par}[Def]{} % Paragraphe.

\title{Transfers on Milnor-Witt K-theory} % Remplacez "Titre" par votre titre.
\author{\sc Niels FELD\footnote{Adress: Institut Fourier, 100 Rue des Mathématiques, Grenoble, France.}
\footnote{E-mail adress: <niels.feld@univ-grenoble-alpes.fr>.}  } % Remplacez "Auteur" par le nom du ou des auteurs.

\date{2020} % Avec l'argument "\today", la date indiquée sera celle du jour de la compilation.

\begin{document} % Début du document proprement dit.

\maketitle % Crée le bandeau de titre.
%\begin{abstract} % Résumé. Supprimer ces lignes si le texte ne s'y prête pas.
% Après quelques rappels concernant la théorie des catégories et celle des topos, on définit et étudie certaines propriété du topos effectif introduit par Hyland dans [?].
%\end{abstract}

%\tableofcontents % Table des matières. À réserver aux documents vraiment longs (au moins 25 pages).

\begin{abstract} We give a new proof of the fact that Milnor-Witt K-theory has geometric transfers. The proof yields to a simplification of Morel's conjecture about transfers on contracted homotopy sheaves.
\end{abstract}


\Codes{14C17, 14C35, 11E81}
\tableofcontents

\pagebreak

\section{Introduction}

\subsection{Current work}

In \cite[Chapter 3]{Mor12}, Morel introduced the Milnor-Witt K-theory of a field. Following ideas of Bass and Tate \cite{BassTate73}, one can define geometric transfer maps
\begin{center}
$\Tr_{x_1,\dots,x_r/E}=\Tr_{x_r/E(x_1,\dots,x_{r-1})}\circ \dots \circ \Tr_{x_1/E}:\kMW_*(E(x_1,\dots, x_r),\omega_{E(x_1,\dots, x_r)/E})
\to \kMW_*(E)$
\end{center}
 on $\kMW$ for finite extensions $E(x_1,\dots, x_r)/E$. Morel proved in \cite[Chapter 4]{Mor12} that any homotopy sheaves of the form $M_{-2}$ admit such transfers and that they are functorial. In particular, this applies to Milnor-Witt K-theory.
\par In this article, we give an alternative proof of this result:

\begin{customthm}{1} [Theorem \ref{KatoMorel}]
The transfers maps
\begin{center}
$\Tr_{x_1,\dots,x_r/E}:\kMW_*(E(x_1,\dots, x_r),\omega_{E(x_1,\dots, x_r)/E})\to \kMW_*(E)$
\end{center}  do not depend on the choice of the generating system $(x_1,\dots , x_r)$.

\end{customthm}
The idea is to reduce to the case of $p$-primary fields (see Definition \ref{DefprimaryFields}) then study the transfers manually, as Kato originally did for Milnor K-theory (see \cite{GS17} for a modern exposition). More elementary, the proof does not apply in full generality to the case of contracted homotopy sheaves $M_{-1}$. However, we obtain as a corollary a reduction of Morel's conjecture \cite[Conjecture 4.1.13]{Fel20}.

\begin{customthm}{2} [Theorem \ref{ReductionConjectureMoinsUn}]
In order to prove that a contracted homotopy sheaf $M_{-1}$ has functorial transfers, it suffices to consider the case of $p$-primary fields (where $p$ is a prime number).
\end{customthm}

\subsection{Outline of the chapter}

In Subsection \ref{subsectionOnprimaryfields}, we recall the basic properties of some fields that we call $p$-primary fields. For $p$ a prime number, a $p$-primary field has no nontrivial finite extension prime to $p$ (see Definition \ref{DefprimaryFields}).
\par In Subsection \ref{subsectionTransfersKMW}, we prove that Milnor-Witt K-theory admit transfer maps that are functorial. The proof is similar to the original proof of Kato for Milnor K-theory (see \cite{GS17}): we reduce to the case of $p$-primary fields then study the transfers manually.

\section*{Notation}\label{Conventions}
Throughout the paper, we fix a (commutative) field $k$ and we assume moreover that $k$ is perfect (of arbitrary characteristic).
\par By a field $E$ over $k$, we mean {\em a finitely generated extension of fields $E/k$}.
\par Let $f:X\to Y$ be a morphism of schemes. Denote by $\LL_f$ (or $\LL_{X/Y}$) the virtual vector bundle over $Y$ associated with the cotangent complex of $f$, and by $\LLL_f$ (or $\LLL_{X/Y}$) its determinant. Recall that if $p:X\to Y$ is a smooth morphism, then $\LL_p$ is (isomorphic to) $\TT_p=\Om_{X/Y}$ the space of (Kähler) differentials. If $i:Z\to X$ is a regular closed immersion, then $\LL_i$ is the normal cone $-\NN_ZX$. If $f$ is the composite $\xymatrix{ Y \ar[r]^i & \PP^n_X \ar[r]^p & X}$  with $p$ and $i$ as previously (in other words, if $f$ is lci projective), then $\LL_f$ is isomorphic to the virtual tangent bundle $i^*\TT_{\PP^n_X/X} - \NN_Y(\PP^n_X) $.
\par Let $X$ be a scheme and $x\in X$ a point, we denote by $\LL_{x}=(\mathfrak{m}_x/\mathfrak{m}_x^2)^{\vee}$ and $\LLL_x$ its determinant. Similarly, let $v$ a discrete valuation on a field, we denote by $\LLL_{v}$ the line bundle $(\mathfrak{m}_v/\mathfrak{m}_v^2)^{\vee}$.
% \par Let $F/E$ be a field extension, we denote by $\Om_{F/E}$ the $F$-vector space of (Kähler) differentials. We use the same notation to denote its canonical representant in the category of virtual vector spaces. Dually, let $X/S$ be a scheme morphism, we denote by $\TT_{X/S}$ the sheaf of modules of differentials (or rather, its virtual representant).
\par Let $E$ be a field (over $k$) and $v$ a valuation on $E$. We will always assume that $v$ is discrete. We denote by $\mathcal{O}_v$ its valuation ring, by $\mathfrak{m}_v$ its maximal ideal and by $\kappa(v)$ its residue class field. We consider only valuations of geometric type, that is we assume: $k\subset \mathcal{O}_v$, the residue field $\kappa(v)$ is finitely generated over $k$ and satisfies $\operatorname{tr.deg}_k(\kappa(v))+1=\operatorname{tr.deg}_k(E)$.

\par Let $E$ be a field. We denote by $\GW(E)$ the Grothendieck-Witt ring of symmetric bilinear forms on $E$. For any $a\in E^*$, we denote by $\ev{a}$ the class of the symmetric bilinear form on $E$ defined by $(X,Y)\mapsto aXY$ and, for any natural number $n$, we put $n_{\epsilon}=\sum_{i=1}^n \ev{-1}^{i-1}$.
\par To any natural number $n$, we can associate an element in $\GW(E)$ denoted by $n_{\epsilon}=\sum_{i=1}^n \ev{-1}^{i-1}$. Recall that, if $n$ and $m$ are two natural numbers, then $(nm)_{\epsilon}=n_{\epsilon}m_{\epsilon}$.

\section*{Acknowledgement}
I deeply thank my two PhD advisors Frédéric D\'eglise and Jean Fasel. This work received support from the French "Investissements d'Avenir" program, project ISITE-BFC (contract ANR-lS-IDEX-OOOB).

\section{Transfers on Milnor-Witt K-theory}

\subsection{On $p$-primary fields}
\label{subsectionOnprimaryfields}

We recall some facts about fields (See \cite[§1]{Shap82} and \cite[Section 5]{BassTate73}). Let $E$ be a field and $p$ a prime number. Fix a separable closure $E_s$ of $E$ and consider the set of all sub-extensions of $E_s$ that contain $E$ and that can be realized as a union of finite prime-to-$p$ extensions  of $E$. Zorn's lemma implies that this set contains a maximal element $E_{\ev{p}}$ for the inclusion.

\begin{Pro}
If $F$ is a finite extension of $E$ contained in $E_{\ev{p}}$, then its degree $[F:E]$ is prime to $p$.
\end{Pro}
\begin{proof}
Write $F=E(x_1,\dots, x_r)$ with $x_i\in F$. Each $x_i$ is contained in a prime-to-$p$ extension of $E$ hence has a degree prime to $p$.
\end{proof}

\begin{Pro}
If $F$ is a finite extension of $E_{\ev{p}}$, then its degree $[F:E_{\ev{p}}]$ is equal to $p^n$ for some natural number $n$.
\end{Pro}
\begin{proof}
Let $x$ be any element in $F$ and denote by $P_x$ its irreducible polynomial over $E_{\ev{p}}$. We prove that its degree is a power of $p$. All the coefficients lie in a finite prime-to-$p$ extension of $E$. If the degree of $x$ over $E_{\ev{p}}$ is prime to $p$, then $E_{\ev{p}}(x)$, which is a nontrivial extension of $E_{\ev{p}}$, contradicts the maximality of $E_{\ev{p}}$. Write $p^nm$ the degree of $x$ over $E_{\ev{p}}$ with $n,m \geq 1$ and $(n,p)=1$. Let $F_N$ be the normal closure of $F$ in $E_s$ ; it is a Galois extension of $E_{\ev{p}}$ whose degree over $E_{\ev{p}}$ is divisible by $p^nm$. If $n\neq 1$, then a Sylow $p$-subgroup $S(p)$ of $\operatorname{Gal}(F_N/E_{\ev{p}})$ is a nontrivial proper subgroup and the fixed field $F_N^{S(p)}$ is a nontrivial prime-to-$p$ extension of $E_{\ev{p}}$, which is absurd. Thus $n=1$ and the result follows.
\end{proof}

The previous result leads to the following definition.
\begin{Def} \label{DefprimaryFields}
A field  that has no nontrivial finite extension prime to $p$ is called $p$-primary.
\end{Def}

\begin{Pro} 
Let $F$ be a nontrivial finite extension of $E_{\ev{p}}$ contained in $E_s$ and let $p^n$ be the degree $[F:E_{\ev{p}}]$. Then there is a tower of fields
\begin{center}
$E_{\ev{p}} = F_1 \subset F_2 \subset \dots \subset F_n=F$ 
\end{center}
such that $[F_i:F_{i-1}]=p$.
\end{Pro}
\begin{proof}
We prove the result by induction on $n$. We need to find a subfield $K$ of $F$ whose degree over $E_{\ev{p}}$ is $p^{n-1}$. The group $G=\operatorname{Gal}(E_s/E_{\ev{p}})$ is a pro-$p$-group since all finite  extensions of $E_{\ev{p}}$ contained in $E_s$ are $p$-power extensions. Galois theory implies that $E$ is the fixed subfield of a subgroup $H$ of $G$ with $[G:H]=p^n$. We will find a subgroup $H_1$, such that $H\subset H_1 \subset G$ and $[G:H_1]=p^{n-1}$. Letting $K=E_s^{H_1}$, we will get the desired subfield $K$.
\par The group $H$ is subgroup of $G$ of finite index hence is open. By the class equation, it also follows that $H$ has only a finite number of conjugates in $G$. Let $H'=\cap_{x\in G} x^{-1}Hx$, then $H'$ is an open normal subgroup of $G$ containing $H$. The group $G/H'$ is a finite $p$-group containing $H/H'$. By the Sylow theorems, we can find $H_1$, normal in $G$, with $H\subset H_1 \subset G$ and $[G:H_1]=p^{n-1}$. This ends to proof. 
\end{proof}

Similarly, we obtain the following result.
\begin{Lem} \label{Lem7.3.7}
Let $p$ be a prime number and $E$ a $p$-primary field. Let $F/E$ be a finite extension.
\begin{enumerate}
\item The field $F$ inherits the property of having no nontrivial finite extension of degree prime to $p$.
\item If $F\neq E$, then there exists a subfield $E\subset F'\subset  F$ such that $F'/E$ is a normal extension of degree $p$.
\end{enumerate}
\end{Lem}

\subsection{Transfers on Milnor-Witt K-theory}
\label{subsectionTransfersKMW}

We refer to \cite[§3]{Mor12} or \cite[§1]{Fel18} for the definitions and basic properties regarding Milnor-Witt K-theory.
Recall the definition of transfers on Milnor-Witt K-theory, this follows from the original definition of Bass-Tate (see \cite{BassTate73}, see also \cite{GS17}).

\begin{The}[Homotopy invariance] \label{ThmHomInvarianceHIbis}
Let $F$ be a field and $F(t)$ the field of rational fractions with coefficients in $F$ in one variable $t$. We have a split short exact sequence
%  \footnote{Recall \ref{KahlerDiffFractionField}}

\begin{center}
$\xymatrix@C=10pt@R=20pt{
0 \ar[r] &  \kMW_*(F) \ar[r]^-{\res}   & \kMW_*(F(t)) \ar[r]^-d  &  \bigoplus_{x\in {(\AAA_F^1)}^{(1)}} \kMW_{*-1}(\kappa(x), \LLL_{x}) \ar[r] & 0
}
$
\end{center}
where $d=\bigoplus_{x\in {(\AAA_F^1)}^{(1)}}\partial_x$ is the usual differential.
\end{The}
 \begin{proof}
See \cite[Theorem 5.38]{Mor12}\footnote{In fact, Morel does not use twisted sheaves but chooses a canonical generator for each $\LLL_x$ instead, which is equivalent.}.
\end{proof}

\begin{Par} \label{BassTateTransfersDefinitionBis}
Let $\phi:E\to F$ be a monogenous finite field extension and choose $x\in F$ such that $F=E(x)$. The homotopy exact sequence implies that for any $\beta\in \kMW_*(F,\LLL_{F/k})$ there exists ${\gamma\in \kMW_*(E(t),\LLL_{E(t)/k})}$ with the property that $d(\gamma)=\beta$. Now the valuation at $\infty$ yields a morphism
 
 \begin{center}
 ${\partial_\infty:\kMW_{*+1}(E(t),\LLL_{E(t)/k})\to
   \kMW_*(E,\LLL_{E/k})}$
   
    \end{center}  which vanishes on the image of $i_*$. We denote by $\phi^*(\beta)$ or by $\operatorname{Tr}_{x/E}(\beta)$ the element $-\partial_{\infty}(\gamma)$; it does not depend on the choice of $\gamma$. This defines a group morphism
  \begin{center}
  
  $\Tr_{x/E}:\kMW_*(E(x),\LLL_{F/k})\to \kMW_*(E,\LLL_{E/k})$
  \end{center}
  called the {\em transfer map} and also denoted by $\operatorname{Tr}_{x/E}$. The following result completely characterizes the transfer maps.

\end{Par}
\begin{Lem} Keeping the previous notations, let
  \begin{center}
  
  $d=(\bigoplus_x d_x) \oplus d_{\infty}:\kMW_{*+1}(E(t),\LLL_{F(t)/k})\to
   (\bigoplus_x\kMW_*(E(x),\LLL_{E(x)/k})\oplus \kMW_*(E,\LLL_{E/k})$
  \end{center}
  be the total twisted residue morphism (where $x$ runs through the set of monic irreducible polynomials in $E(t)$). Then, the transfer maps $\operatorname{Tr}_{x/E}$ are the unique morphisms such that $\sum_x ( \operatorname{Tr}_{x/E}\circ d_x)+d_\infty=0$.
  \end{Lem}
  \begin{proof}
Straightforward (see \cite[§4.2]{Mor12}).
% or \cite[Lemma 2.10]{CalFasel18}.
  \end{proof}

\begin{Def}
Let $F=E(x_1,x_2,\dots, x_r)$ be a finite extension of a field $E$ and consider the chain of subfields
\begin{center}

$E\subset E(x_1)\subset E(x_1,x_2)\subset \dots \subset E(x_1,\dots , x_r)=F.$
\end{center}
Define by induction:
\begin{center}
$\Tr_{x_1,\dots,x_r/E}=\Tr_{x_r/E(x_1,\dots,x_{r-1})}\circ \dots \circ \Tr_{x_2/E(x_1)}\circ \Tr_{x_1/E}$
\end{center}
\end{Def}

We give an elementary proof of the fact that the definition does not depend on the choice of the factorization (see 	\cite[Theorem 4.27]{Mor12} for the original proof):

\begin{The} \label{KatoMorel}
The maps $\Tr_{x_1,\dots,x_r/E}:\kMW_*(F)\to \kMW_*(E)$ do not depend on the choice of the generating system $(x_1,\dots , x_r)$.

\end{The}

We begin with a series of lemmas aimed at reducing the theorem to the case of $p$-primary fields.

\begin{Lem}\label{LamVII2.2}
Let $F=E(x)/E$ be a simple extension of degree $n$ of characteristic zero fields and consider the transfer map $\Tr_{F/E}:\GW(F)\to \GW(E)$. If $n$ is odd, then
\begin{center}
$\Tr_{F/E}(1)=n_\epsilon.$
\end{center}
If $n$ is even, then
\begin{center}
$\Tr_{F/E}(1)=(n-1)_\epsilon+\langle - N_{F/E}(x) \rangle$.
\end{center}
where $N_{F/E}(x)$ is the classical norm of $x\in F^\times$.
\end{Lem}
\begin{proof}
See \cite[VII.2.2]{Lam05}.
\end{proof}

\begin{Lem} 
Let $F/E$ be a finite extension of degree $n$ of characteristic zero fields and consider the transfer map $\Tr_{F/E}:\GW(F)\to \GW(E)$. If $n$ is odd, then
\begin{center}
$\Tr_{F/E}(1)=n_\epsilon.$
\end{center}
If $n$ is even, then there exist $a_1,\dots a_n\in E^\times$ such that
\begin{center}
$\Tr_{F/E}(1)=\sum_i \ev{a_i}.$
\end{center}
\end{Lem}
\begin{proof}
When the extension is simple, this is Lemma \ref{LamVII2.2}. We conclude by induction on the number of generators.
\end{proof}

\begin{Lem}\label{RiousLemma}
Let $E$ be a field of characteristic $p>0$. Let $\alpha \in \GW(E) $ be an element in the kernel of the rank morphism $\GW(E)\to \ZZ$. Then $\alpha$ is nilpotent in $\GW(E)$.
\end{Lem}
\begin{proof}
(See \cite[Lemma B.4]{LYZ16}) As the set of nilpotent elements in the commutative ring $\GW(E)$ is an ideal, we may assume
 $\alpha =\ev{t}-1$ where $t\in E^{\times}$. We have $(1 +\alpha)^2=\ev{t^2}= 1$, so that $\alpha^2=-2\alpha$. 
  By induction, we get $\alpha^n= (-2)^{n-1}\alpha$ for $n \geq 1$: 
  we have to show that $\alpha$ is annihilated by a power of two. If 
  $p= 2$, $2\alpha= 0$ holds (see \cite[Lemma 3.9]{Mor12}), i.e. 
  $\alpha^2= 0$. Now we assume $p\geq 3$ so that there is no danger 
  thinking in terms of usual quadratic forms. We first consider $\mu 
  :=\ev{-1}-1\in \GW(\mathbb{F}_p)$.  The quadratic form $-x^2-y^2$ 
  over $\mathbb{F}_p$ represents $1$ (see \cite[Proposition 
    4,§IV.1.7]{Serre77}) so that $\ev{-1}+\ev{-1}=\ev{1}+\ev{1}\in \GW(\mathbb{F}_p)$, 
  i.e. $2\mu= 0 \in \GW(\mathbb{F}_p)$,which gives $\mu^2=0$. Let 
  $t\in E^{\times}$ be any nonzero element in an extension $E$ of 
  $\mathbb{F}_p$. The quadratic form $q(x,y) :=x^2-y^2= (x+y)(x-y)$ 
  represents $t$ (this is $q((1+t)/2,(1-t)/2)$), which easily implies that 
  $\ev{1}+\ev{-1}=\ev{t}+\ev{-t}$. This is equivalent to saying $(2 
  +\mu)\alpha= 0 \in \GW(E)$. It follows that $4\alpha = (2-\mu)(2 
  +\mu)\alpha= 0$, and then $\alpha^3= 0$.
\end{proof}

\begin{Lem} \label{CorQuadraticBezout}
Consider two finite extensions $F/E$ and $L/E$ of coprime degrees $n$ and $m$, respectively. Let $x\in \kMW(E)$ such that $\res_{F/E}(x)=0=\res_{L/E}(x)$. Then $x=0$.
\end{Lem}
\begin{proof}
Applying the transfer map to $\res_{F/E}(x)$ and $\res_{L/E}(x)$, we see that $x$ is killed by $\Tr_{F/E}(1)$ and $\Tr_{L/E}(1)$. In characteristic zero,up to swapping $n$ and $m$, we may assume that $n$ is odd, hence $\Tr_{F/E}(1)=n_{\epsilon}$ and $\Tr_{L/E}(1)=\sum_i \ev{a_i}$ for some $a_1,\dots, a_m\in E^\times$. Write $n=2r+1$. There exist $a,b\in \ZZ$ such that $an+bm=r$ since $n$ and $m$ are coprime. Recall that the hyperbolic form $h=1+\ev{-1}$ satisfies $\ev{a_i}h=h$ for any $i$ (see \cite[Lemma 3.7]{Mor12}). Hence $rh=(an_{\epsilon}+b\sum_i \ev{a_i})h$ and $1=n_{\epsilon}-rh=(1-ah)\Tr_{F/E}(1)-bh\Tr_{L/E}(1)$ kills $x$.
\par In characteristic $p>0$, there exist two nilpotent $\alpha$ and $\alpha'$ in $ \GW(E)$ such that $\Tr_{F/E}(1)=n+\alpha$ and $\Tr_{L/E}(1)=m+\alpha'$, according to Lemma \ref{RiousLemma}. Hence for a natural number $s$ large enough, the element $x$ is killed by the coprime numbers $n^{p^s}$ and $m^{p^s}$ so that $x=0$.
\end{proof}

\begin{Lem}\label{FundamentalLemma}
Let $E$ be a field of characteristic $p>0$. Let $F_1,\dots, F_n$ be finite extensions of coprime degrees $d_1,\dots,d_n$. Let $\delta \in \kMW_*(E)$ be an element such that $\res_{F_i/E}(\delta)=0$ for any $i$. Then, $\delta$ is zero. 
\end{Lem}
\begin{proof} In zero characteristic, this follows as in Lemma \ref{CorQuadraticBezout}. Assume the characteristic of $E$ to be nonzero.
Let $1\leq i\leq n$, the projection formula proves that $\delta$ is killed by $\Tr_{F_i/E}(1)$. Thus, according to Lemma \ref{RiousLemma}, there exist a nilpotent element $\alpha_i$ in $\GW(E)$ such that $d_i+\alpha_i$ kills $\delta$. Since the degrees $d_i$ are coprime, a Bezout combination yields a nilpotent element $\alpha'$ in $\GW(E)$ such that $1+\alpha'$ kills $\delta$. Finally, we can find a natural number $n$ large enough such that $\delta=1+(\alpha')^{p^n}\cdot \delta=(1+\alpha')^{p^n}\cdot \delta=0$.
\end{proof}

\begin{Lem} \label{R3a}
Let $F/E$ be a field extension and $w$ be a valuation on $F$ which restricts to a non trivial valuation $v$ on $E$ with ramification $e$. We have a commutative square
\begin{center}
$\xymatrix{
\kMW_*(E)
\ar[r]^-{\partial_v}
\ar[d]_{\res_{F/E}}
&
\kMW_{*-1}(\kappa(v),\LLL_v)
\ar[d]^{e_{\epsilon}\cdot \res_{\kappa(w)/\kappa(v)}}
\\
\kMW_*(F)
\ar[r]_-{\partial_w}
&
\kMW_{*-1}(\kappa(w),\LLL_w)
}$
\end{center} 
where $e_{\epsilon}=\sum_{i=1}^{e}\ev{-1}^{i-1}$.

\end{Lem}
\begin{proof}
See \cite[Lemma 3.19]{Mor12}.
\end{proof}

\begin{Lem}\label{R1c_faible}
Let $F/E$ be a field extension and $x\in {(\AAA_E^1)}^{(1)}$ a closed point. Then the following diagram
\begin{center}
$\xymatrixcolsep{5pc}\xymatrix{
\kMW_*(E(x),\LLL_{E(x)/k}) \ar[r]^-{\Tr_{x/E}}   \ar[d]_{\oplus_y \res_{F(y)/E(x)}} 
&
\kMW_*(E,\LLL_{E/k}) \ar[d]^{\res_{F/E}} 
\\
\bigoplus_{y\mapsto x} \kMW_*(F(y),\LLL_{F(y)/k}) \ar[r]_-{\sum_{y}e_{y,\epsilon}\Tr_{y/F}} & \kMW_*(F,\LLL_{F/k}) 
}$
\end{center}
is commutative, where  the notation $y\mapsto x$ stands for the closed points of $ {\AAA_F^1}$ lying above $x$, and $e_{y,\epsilon}=\sum_{i=1}^{e_y}\ev{-1}^{i-1}$ is the quadratic form associated to the ramification index of the valuation $v_y$ extending $v_x$ to $F(t)$.
\end{Lem}

\begin{proof}
According to Lemma \ref{R3a}, the following diagram
\begin{center}
$\xymatrix{
\kMW_*(E(t)) \ar[r]^{\partial_x}  \ar[d]_{\res_{F(t)/E(t)}}
& \kMW_{*-1}(E(x),\LLL_x) \ar[d]^{\oplus_y e_{y,\epsilon} \res_{F(y)/E(x)}} 
\\
\kMW_*(F(t)) \ar[r]_-{\oplus_y \partial_y} & \bigoplus_{y\mapsto x} 
\kMW_{*-1}(F(y),\LLL_y)
}$
\end{center} 
is commutative hence so does the diagram
\begin{center}
$\xymatrix{
\kMW{*}(E(t))   \ar[d]_{\res_{F(t)/E(t)}}
& \kMW{*}(E(x),\LLL_x) \ar[d]^{\oplus_y e_{y,\epsilon} \res_{F(y)/E(x)}} \ar[l]_{\rho_x} \\
\kMW{*}(F(t))  &
 \bigoplus_{y\mapsto x} \kMW{*}(F(y),\LLL_y)\ar[l]^-{\oplus_y \rho_y}
}$
\end{center}where $\rho_x$ is the canonical splitting (see Theorem \ref{ThmHomInvarianceHIbis}). Then, we conclude according to the definition of the Bass-Tate transfer maps \ref{BassTateTransfersDefinitionBis}.
\end{proof}

\begin{Rem}
The multiplicities $e_y$ appearing in the previous lemma are equal to 
\begin{center}
${[E(x):E]_i/[F(y):F]_i}$
\end{center} 
where $[E(x):E]_i$ is the inseparable degree.
\end{Rem}

\begin{The} [Strong R1c] \label{R1c_fort_KMW}
Let $E$ be a field, $F/E$ a finite field extension and $L/E$ an arbitrary field extension. Write $F=E(x_1,\dots, x_r)$ with $x_i\in F$, $R=F\otimes_E L$ and $\psi_{\mathfrak{p}}:R\to R/\mathfrak{p}$ the natural projection defined for any $\mathfrak{p}\in \Spec(R)$. Then the diagram
\begin{center}
$\xymatrixcolsep{9pc}\xymatrix{
\kMW_*(F,\LLL_{F/k}) \ar[r]^{\Tr_{x_1,\dots, x_r/E}} \ar[d]_{\oplus_{\mathfrak{p}} \res_{(R/\mathfrak{p})/F}}
& \kMW_*(E,\LLL_{E/k})  \ar[d]^{\res_{L/E}} \\
\bigoplus_{\mathfrak{p}\in \Spec(R)}\kMW_*(R/\mathfrak{p},\LLL_{(R/\mathfrak{p})/k}) 
\ar[r]_-{
\sum_{\mathfrak{p}}  (m_p)_{\epsilon} \Tr_{
\psi_{\mathfrak{p}}(a_1),
\dots,
\psi_{\mathfrak{p}}(a_r)/L}}
& \kMW_*(L,\LLL_{L/k})
}$
\end{center}
is commutative where $m_p$ the length of the localized ring $R_{(p)}$.
\end{The}

\begin{proof} We prove the theorem by induction. For $r=1$, this is Lemma \ref{R1c_faible}. Write $E(x_1)\otimes_E L=\prod_j R_j$ for some Artin local $L$-algebras $R_j$, and decompose the finite dimensional $L$-algebra $F\otimes_{E(x_1)} R_j$ as $F\otimes_{E(x_1)} R_j=\prod_{i} R_{ij}$ for some local $L$-algebras $R_{ij}$. We have $F\otimes_E L\simeq \prod_{i,j}R_{ij}$. Denote by $L_j$ (resp. $L_{ij}$) the residue fields of the Artin local $L$-algebras $R_j$ (resp. $R_{ij}$), and $m_j$ (resp. $m_{ij}$) for their geometric multiplicity. We can conclude as the following diagram commutes

\begin{center}
\scalebox{0.88}{$\xymatrixcolsep{10.5pc}\xymatrix{
\M(F,\LLL_{F/k}) \ar[r]^{\tr_{x_1,\dots, x_r/E}} \ar[d]_{\oplus_{ij} \res_{L_{ij}/F}}
&
\M(E(x_1),\LLL_{E(x_1)/k}) 
\ar[d]^{\oplus \res_{L_j/E(x_1)}}
\ar[r]^{\tr_{x_1/E}}
& \M(E,\LLL_{E/k})  
\ar[d]^{\res_{L/E}} 
\\
\bigoplus_{ij}\M(L_{ij},\LLL_{L_{ij}/k}) 
\ar[r]_-{
\sum_{ij}  (m_{ij}m_j^{-1})_{\epsilon} \tr_{
\psi_{ij}(x_1),
\dots,
\psi_{ij}(x_r)/L_j}}
&
\bigoplus_j \M(L_j,\LLL_{L_j/k}) 
\ar[r]_{\sum_j (m_j)_{\epsilon}\tr_{\psi_j(x_1)/L}}
& \M(L,\LLL_{L/k})
}$
}

\end{center}
since both squares are commutative by the inductive hypothesis and the multiplicity formula $(mn)_{\epsilon}=m_{\epsilon}n_{\epsilon}$ for any natural numbers $m,n$.
\end{proof}

\begin{comment}

We prove the theorem by induction. For $r=1$, this is Lemma \ref{R1c_faible}. Write $E(x_1)\otimes_E L\sum R_j$ for some local $L$-algebras $R_j$, and decompose the finite dimension $L$-algebra $K\otimes_{E(x_1)} R_j$ as $F\otimes_{E(x_1)} R_j=\bigoplus_{i} R_{ij}$ for some local $L$-algebras $R_{ij}$. We have $F\otimes_E L\simeq _{i,j}R_{ij}$. Denote by $L_j$ (resp. $L_{ij}$) the residue fields of the $L$-algebras $R_j$ (resp. $R_{ij}$), and $m_j$ (resp. $m_{ij}$) for the corresponding multiplicity. The following diagram commutes
\begin{center}
$\xymatrixcolsep{9pc}\xymatrix{
\M(F) \ar[r]^{\Tr_{x_1,\dots, x_r/E}} \ar[d]_{\oplus_{ij} \res_{L_{ij}/F}}
&
\M(E(x_1)) 
\ar[d]^{\oplus \res_{L_j/E(x_1)}}
\ar[r]^{\Tr_{x_1/E}}
& \M(E)  
\ar[d]^{\res_{L/E}} 
\\
\bigoplus_{ij}\M(L_{ij}) 
\ar[r]_-{
\sum_{ij}  (m_{ij}m_j^{-1})_{\epsilon} \Tr_{
\psi_{ij}(x_1),
\dots,
\psi_{ij}(x_r)/L_j}}
&
\bigoplus_j \M(L_j) 
\ar[r]_{\sum_j (m_j)_{\epsilon}\Tr_{\psi_j(x_1)/L}}
& \M(L)
}$

\end{center}
since both squares are commutative by the inductive hypothesis.

\end{comment}
\begin{The}
Assume that Theorem \ref{KatoMorel} holds for all fields $p$-primary $E$ for any prime number $p$.
Then the theorem holds for any field $E$.
\end{The}

\begin{proof}Consider two decompositions 
\begin{center}
$E\subset E(x_1)\subset E(x_1,x_2)\subset \dots \subset E(x_1,\dots , x_r)=F.$

\end{center}
and
\begin{center}
$E\subset E(y_1)\subset E(y_1,y_2)\subset \dots \subset E(y_1,\dots , y_s)=F.$

\end{center}
of $F$. Let $\alpha\in \M(F)$ and denote by $\delta$ the element $\Tr_{x_1,\dots,x_r/E}(\alpha)-\Tr_{y_1,\dots y_s/E}(\alpha)$. 
Fix $p$ a prime number and let $L$ be a maximal prime to $p$ extension of $E$ ($L$ has no nontrivial finite extension of degree prime to $p$). With the notation of Theorem \ref{R1c_fort_KMW}, the map $\sum_{\mathfrak{p}}  (m_p)_{\epsilon} \Tr_{
\psi_{\mathfrak{p}}(x_1),
\dots,
\psi_{\mathfrak{p}}(x_r)/L}$ does not depend on the choice of $x_i$ according to the assumption. Hence $\res_{L/E}(\delta)=0$ and we can find a finite extension $L_p/E$ of degree prime to $p$ such that $\res_{L_p/E}(\delta)=0$. Since this is true for all prime number $p$, we see that the assumption of Lemma \ref{FundamentalLemma} are satisfied. Thus $\delta=0$ and the theorem is proved. 
\end{proof}

\begin{Rem} More generally, one may replace $\kMW_*$ by any contracted homotopy sheaf $M_{-1}$ and apply the proof verbatim. In particular, we have the following simplification of Morel's conjecture \cite[Conjecture 4.1.13]{Fel20}.
\end{Rem}

\begin{The} \label{ReductionConjectureMoinsUn}
In order to prove that a contracted homotopy sheaf $M_{-1}$ has functorial transfers, it suffices to consider the case of $p$-primary fields (where $p$ is a prime number).
\end{The}

\begin{Pro}[Bass-Tate-Morel Lemma] \label{BassTateMorelLemma} Let $F(x)$ be a monogenous extension of $F$. Then $\kMW_*(F(x))$ is generated as a left $\kMW_*(F)$-module by elements of the form
\begin{center}
$\eeta^m\cdot [p_1(x),p_2(x),\dots, p_n(x)]$
\end{center}
where the $p_i$ are monic irreducible polynomials of $F[t]$ satisfying 
\begin{center}
${\deg(p_1)< \deg(p_2) < \dots < \deg(p_n)\leq d-1}$
\end{center} where $d$ is the degree of the extension $F(x)/F$.
\end{Pro}
\begin{proof}
Straightforward computations (see also \cite[Lemma 3.25.1]{Mor12}).
\end{proof}

\begin{Cor} \label{CorBassTateMorel}
Let $F/E$ be a finite field extension and assume one of the following conditions holds:
\begin{itemize}
\item $F/E$ is a quadratic extension,
\item $F/E$ is a prime degree $p$ extension and $E$ has no nontrivial extension of degree prime to $p$.
\end{itemize}

Then $\kMW_*(F)$ is generated as a left $\kMW_*(E)$-module by $F^\times$. 
\end{Cor}
\begin{proof}

In both cases, the extension $F/E$ is simple and the only monic irreducible polynomial in $E[t]$ of degree strictly smaller than $[F:E]$ are the polynomials of degree 1. We conclude by Proposition \ref{BassTateMorelLemma}.
\end{proof}

{\em In the following, we fix a prime number $p$ and $E$ a $p$-primary field.}

\begin{Pro} \label{Prop7.3.8}
Let $F=E(x)$ be a monogenous extension of $E$ of degree $p$. Then the transfers $\Tr_{x/E}:\kMW_*(E(x),\LLL_{E(x)/k})\to \kMW_*(E,\LLL_{E/k})$ do not depend on the choice of $x$.
\end{Pro}

\begin{proof}
According to Lemma \ref{CorBassTateMorel}, the group $\kMW_*(F,\LLL_{F/k})$ is generated by products of the form $\res_{F/E}(\alpha)\cdot [\beta]$ with $\alpha \in \kMW_*(E,\LLL_{E/k})$ and $\beta\in  F^\times$. The projection formula yields
\begin{center}
$\Tr_{x/E}(\res_{F/E}(\alpha)\cdot [\beta])=
\alpha\cdot \Tr_{F/E}([\beta])$
\end{center}
which does not depend on a $x$ (see \cite[§1]{Fasel18bis}).
\end{proof}
\begin{Par} We may now use the notation $\Tr_{F/E}:\kMW_*(F,\LLL_{F/k})\to \kMW_*(E,\LLL_{E/k})$ for extensions of prime degree $p$.
\end{Par}

\begin{Pro} \label{Mor12Remark5.20}
Let $F$ be a field complete with respect to a discrete valuation $v$, and $F'/F$ a normal extension of degree $p$. Denote by $v'$ the unique extension of $v$ to $F'$. Then the diagram
\begin{center}
$\xymatrix{
\kMW_*(F',\LLL_{F'/k})
\ar[r]^-{\partial_{v'}}
\ar[d]_{\Tr_{F'/F}}
&
\kMW_{*-1}(\kappa(v'),\LLL_{\kappa(v')})
\ar[d]^{\Tr_{\kappa(v')/\kappa(v)}}
\\
\kMW_*(F)
\ar[r]_-{\partial_v}
&
\kMW_{*-1}(\kappa(v))
}$
\end{center}
is commutative.
\end{Pro}

\begin{proof}
See \cite[Remark 5.20]{Mor12}.
\end{proof}

\begin{Cor} \label{Mor12Remark5.20Cor}
Let $F/E$ be a normal extension of degree $p$ and let $x\in (\AAA^1_E)^{(1)}$. Then the diagram
\begin{center}
$\xymatrix{
\kMW_*(F(t),\LLL_{F(t)/k})
\ar[r]^-{\oplus \partial_y}
\ar[d]_{\Tr_{F(t)/E(t)}}
& 
\bigoplus_{y\mapsto x} \kMW_{*-1}(\kappa(y),\LLL_{\kappa(y)/k})
\ar[d]^{\sum \Tr_{\kappa(y)/\kappa(x)}}
\\
\kMW_*(E(t),\LLL_{E(t)/k})
\ar[r]_{\partial_x}
&
\kMW_{*-1}(\kappa(x),\LLL_{\kappa(x)/k})
}$
\end{center}
is commutative, where $y\to x$ denotes the set of elements $y\in (\AAA^1_F)^{(1)}$ mapping to $x$ through the canonical morphism.

\end{Cor}

\begin{proof}
Denote by $\hat{E}_x$ (resp. $\hat{F}_y$) the completions of $E(t)$ (resp. $F(t)$) with respect to the valuations defined by $x$ (resp. $y$). Consider the following diagram
\begin{center}
$\xymatrix{
\kMW_*(F(t),\LLL_{F(t)/k}) 
\ar[d]_{\Tr_{F(t)/E(t)}}
 \ar[r]
&
\bigoplus_{y\to x} \kMW_{*}(\hat{F}_y,\LLL_{\hat{F}_y/k}) 
\ar[r]^-{\oplus \partial_y} 
\ar[d]^{\sum \Tr_{\hat{F}_y/\hat{F}_x}}
&
\bigoplus_{y\to x} \kMW_{*-1}(\kappa(y),\LLL_{\kappa(y)/k})
\ar[d]^{\sum \Tr_{\kappa(y)/\kappa(x)}}
\\
\kMW_*(E(t),\LLL_{E(t)/k}) \ar[r]
&
\kMW_*(\hat{E}_x,\LLL_{\hat{E}_x/k}) 
 \ar[r]^-{\partial_x}
&
\kMW_{*-1}(\kappa(x),\LLL_{\kappa(x)/k}).
}$
\end{center}
The left-hand square is commutative according to Theorem \ref{R1c_fort_KMW}. The right-hand square commute according to Proposition \ref{Mor12Remark5.20}. Hence the corollary.
\end{proof}

\begin{Lem} \label{Lem7.3.12}
Let $L/E$ be a normal extension of degree $p$, and let $E(a)/E$ be a monogenous finite extension. Assume that $L$ and $E(a)$ are both subfields of some algebraic extension of $E$, and denote by $L(a)$ their composite. Then the following diagram
\begin{center}
$\xymatrix{
\kMW_*(L(a),\LLL_{L(a)/k})
\ar[d]_-{\Tr_{L(a)/E(a)}}
\ar[r]^{\Tr_{a/L}}
&
\kMW_*(L,\LLL_{L/k})
\ar[d]^{\Tr_{L/E}}
\\
\kMW_*(E(a),\LLL_{E(a)/k})
\ar[r]_-{\Tr_{a/E}}
&
\kMW_*(E,\LLL_{E/k})
}$
\end{center}
is commutative. 
\end{Lem}
\begin{proof}
Let $x$ (resp $y_0$) be the closed point of $\AAA^1_E$ (resp. $\AAA^1_L$) defined by the minimal polynomial of $a$ over $E$ (resp. $L$). Given $\alpha \in \kMW_*(L(a),\LLL_{L(a)/k})$, we have $\Tr_{a/L}(\alpha)= -\partial_{\infty}(\beta)$ for some $\beta\in \kMW_{*+1}(L(t),\LLL_{L(t)/k})$ satisfying $\partial_{y_0}(\beta)=\alpha$ and $\partial_{\beta}=0$ for $y\neq y_0$. By Corollary \ref{Mor12Remark5.20Cor}
\begin{center}
$\partial_x(\Tr_{L(t)/E(t)}(\beta))=\sum_{y\mapsto x} \Tr_{\kappa(y)/\kappa(x)}(\partial_y(\beta))=\Tr_{\kappa(y_0)/\kappa(x)}(\alpha)$,
\end{center}
and, similarly, $\partial_{x'}(\Tr_{L(t)/E(t)}(\beta))=0$ for $x\neq x'$. Hence by definition of the transfer map $\Tr_{a/E}$ we have

\begin{center}

$\Tr_{a/E}(\Tr_{L(a)/E(a)}(\alpha))=-\partial_{\infty}(\Tr_{L(t)/E(t)}(\beta))$.
\end{center}
Moreover, since the only point of $\PP^1_L$ above $\infty$ is $\infty$, another application of Corollary \ref{Mor12Remark5.20Cor} gives
\begin{center}
$\partial_\infty(\Tr_{L(t)/E(t)}(\beta))=\Tr_{L/E}(\partial_{\infty}(\beta)).$
\end{center}
Hence the result.
\begin{center}
$ \Tr_{a/E}(\Tr_{L(a)/E(a)}(\alpha))=-\Tr_{L/E}(\partial_{\infty}(\beta)=\Tr_{L/E}(\Tr_{a/L}(\alpha))$.
\end{center}
\end{proof}

\begin{proof}[Proof of Theorem \ref{KatoMorel}]
We keep the previous notations. We already know that it suffices to treat the case when $E$ has no nontrivial extension of degree prime to $p$. Let $p^m$ be the degree of the extension $F/E$. We prove the result by induction on $m$. The case $m=1$ follows from Proposition \ref{Prop7.3.8}. Consider two decompositions 
\begin{center}
$E\subset E(x_1)\subset E(x_1,x_2)\subset \dots \subset E(x_1,\dots , x_r)=F.$

\end{center}
and
\begin{center}
$E\subset E(y_1)\subset E(y_1,y_2)\subset \dots \subset E(y_1,\dots , y_s)=F.$

\end{center}
of $F$. By Lemma \ref{Lem7.3.7}, the extension $E(x_1)/E$ contains a normal subfield $E(x_1')$ of degree $p$ over $E$. Applying Lemma \ref{Lem7.3.12} with $a=x_1$ and $L=E(x'_1)$ yields $\Tr_{x_1/E}=\Tr_{x'_1/E}\circ \Tr_{x_1/E(x'_1)}$. Hence, without loss of generality, we may assume that $x_1=x_1'$ and, similarly, $[E(y_1):E]=p$. Write $F_0$ for the composite of the fields $E(x_1)$ and $E(y_1)$ in $F$ and write $F=F_0(z_1,\dots, z_t)$ with $z_i\in F$. The fields $E(x_1)$ and $E(y_1)$ have no nontrivial prime to $p$ extension, thus we may conclude by the induction hypothesis that the triangles
\begin{center}
$\xymatrixcolsep{6pc}\xymatrix{
\kMW_*(F,\LLL_{F/k})
\ar[r]^-{\Tr_{x_2,\dots, x_r/E(x_1)}}
\ar[d]_{\Tr_{z_1,\dots , z_t/F_0}}
&
\kMW_*(E(x_1),\LLL_{E(x_1)/k})
\\
\kMW_*(F_0,\LLL_{F_0/k})
\ar[ru]_-{\,\,\,\,\, \Tr_{F_0/E(x_1)}}
}$
\end{center}
and
\begin{center}
$\xymatrixcolsep{6pc}\xymatrix{
\kMW_*(F,\LLL_{F/k})
\ar[r]^-{\Tr_{y_2,\dots, y_s/E(y_1)}}
\ar[d]_{\Tr_{z_1,\dots , z_t/F_0}}
&
\kMW_*(E(y_1),\LLL_{E(y_1)/k})
\\
\kMW_*(F_0,\LLL_{F_0/k})
\ar[ru]_-{\,\,\,\,\, \Tr_{F_0/E(y_1)}}
}$
\end{center}
are commutative.

Moreover,  Lemma \ref{Lem7.3.12} for $a=x_1$ and $L=E(y_1)$ implies that the following diagram
\begin{center}
$\xymatrix{
\kMW_*(F_0,\LLL_{F_0/k}) \ar[r]^-{\Tr_{F_0/E(x_1)}}
\ar[d]_-{\Tr_{F_0/E(y_1)}}
&
\kMW_*(E(x_1),\LLL_{E(x_1)/k}) \ar[d]^-{\Tr_{x_1/E}} 
\\
\kMW_*(E(y_1),\LLL_{E(y_1)/k}) \ar[r]_-{\Tr_{y_1/E}}
&
\kMW_*(E,\LLL_{E/k})
}$
\end{center}
is commutative. Putting everything together, we conclude that $\Tr_{x_1,\dots, x_r/E}=\Tr_{y_1,\dots, y_s/E}$.
\end{proof}

  \bibliographystyle{alpha}
  \bibliography{exemple_biblio}

%\bibliographystyle{plain} % D'autres styles sont disponibles. Notez que les distributions LaTeX n'incluent généralement pas de styles de bibliographies francisés ; vous aurez donc des bibliographies en anglais.
%\bibliography{biblioo} % Remplacer "biblio" par le nom de votre fichier de références bibliographiques.

\end{document}